\documentclass{amsart}
\usepackage{graphicx}
\usepackage{amssymb}
\usepackage{amsfonts}
\usepackage{hyperref}
\usepackage{mathrsfs}
\usepackage[margin=3.3cm]{geometry}

  [1995/08/01 v0.1 Double stroke roman fonts]

\def\ds@whichfont{dsrom}
\relax

\DeclareMathAlphabet{\mathds}{U}{\ds@whichfont}{m}{n}

\swapnumbers
\newtheorem{theorem}{Theorem}[section]
\newtheorem{lemma}[theorem]{Lemma}

\newtheorem{proposition}[theorem]{Proposition}
\theoremstyle{definition}

\newtheorem{assumption}[theorem]{Assumption}
\newtheorem{remark}[theorem]{Remark}

\newtheorem*{acknowledgment}{Acknowledgment}

\numberwithin{equation}{section}
\theoremstyle{plain}

\numberwithin{equation}{section} 
\numberwithin{figure}{section} 
\theoremstyle{plain}
\theoremstyle{plain}
\theoremstyle{remark}
\newtheorem*{acknowledgement*}{Acknowledgement}

\newcommand{\cA}{{\mathcal A}}
\newcommand{\cB}{{\mathcal B}}

\newcommand{\cD}{{\mathcal D}}

\newcommand{\cF}{{\mathcal F}}

\newcommand{\cX}{{\mathcal X}}
\newcommand{\cY}{{\mathcal Y}}

\newcommand{\te}{{\theta}}

\newcommand{\Om}{{\Omega}}
\newcommand{\om}{{\omega}}
\newcommand{\ve}{{\varepsilon}}
\newcommand{\del}{{\delta}}

\newcommand{\Gam}{{\Gamma}}

\newcommand{\sig}{{\sigma}}
\newcommand{\al}{{\alpha}}

\newcommand{\bbE}{{\mathbb E}}

\newcommand{\bbN}{{\mathbb N}}
\newcommand{\bbP}{{\mathbb P}}
\newcommand{\bbR}{{\mathbb R}}

\begin{document}
\title[]{An almost sure invariance principle for some classes of non-stationary mixing sequences}
 \author{Yeor Hafouta \\
\vskip 0.1cm
Department of Mathematics\\
The Ohio State University}
\email{yeor.hafouta@mail.huji.ac.il, hafuta.1@osu.edu}

\maketitle
\markboth{Y. Hafouta}{Almost sure invariance principle}
\renewcommand{\theequation}{\arabic{section}.\arabic{equation}}
\pagenumbering{arabic}

\begin{abstract}
In this note we (in particular) prove an almost sure invariance principle (ASIP) for non-stationary and uniformly bounded sequences of random variables which are exponentially fast  $\phi$-mixing. The obtained rate is of order $o(V_n^{\frac14+\del})$ for an arbitrary $\del>0$, where $V_n$ is the variance  
of the underlying partial sums $S_n$. For certain classes of inhomogeneous Markov chains we also prove  a vector-valued ASIP with similar rates.
\end{abstract}

\section{Introduction}\label{Intro}
The central limit theorem (CLT) for partial sums $S_n=\sum_{j=1}^{n}X_j$ of stationary real-valued random variables $\{X_j\}$, exhibiting some type of ``weak dependence", is one of the main topics in probability theory, stating that $(S_n-\bbE[S_n])/\sqrt{V_n},\, V_n=\text{Var}(S_n)$ converges in distribution towards a standard normal random variable.
The almost sure invariance principle (ASIP) is a stronger result stating that there is a coupling between $\{X_j\}$ and  a standard Brownian motion $(W_t)_{t\ge 0}$ such that 
\[
\left|S_n-\bbE[S_n]-W_{V_n}\right|=o(V_n^\frac12),\,\,\text{almost surely}
\]
where $W_{V_n}$ is the value of the Brownian motion at time $t=V_n$.
Both the CLT and the ASIP have corresponding versions for vector-valued sequences.
The ASIP yields, for instance, the functional central limit theorem and the law of iterated logarithm (see \cite{PS}).
While such results are well established for stationary sequences (see, for instance, \cite{PS}, \cite{BP}, \cite{Shao}, \cite{Rio}, \cite{PelASIP} and \cite{GO} and references therein), in the non-stationary case much less is known, especially when the variance (or the covariance matrix) of $S_n$ grows sub-linearly fast in $n$. For instance, in \cite{WZ} a vector-valued ASIP was obtained under conditions guaranteeing that the covariance matrix grows linearly fast. Similar results were obtained for random dynamical systems in \cite{DFGTV1} and \cite{DH}, and the ASIP for elliptic Markov chains in random dynamical environment can be obtained similarly. For these models the variance  (or the covariance matrix) of the underlying partial sums $S_n$ grows linearly fast in $n$ as well, while in \cite{Hyd} a real-valued ASIP was obtained for time-dependent hyperbolic dynamical systems under the assumption that $\text{Var}(S_n)$ grows faster than $n^{\frac12}$.

In this paper we prove the ASIP for non-stationary, uniformly bounded, real or vector valued  exponentially fast $\al$-mixing sequences of random variables\footnote{We will also assume that $\lim_{n\to\infty}\phi(n)<\frac12$, were $\phi(\cdot)$ are the, so-called, $\phi$-mixing coefficients, so the result holds true when $\phi(n)$ decays exponentially fast.}.
Under a certain assumption, which always holds true for real-valued sequences,
we obtain the ASIP with rate $o(s_n^{\frac14+\del})$ for an arbitrary $\del>0$, where in the real-valued case $s_n=V_n=\text{Var}(S_n)$, while in the vector-valued case\footnote{Where $|u|$ is the standard Euclidean norm of a vector and $u\cdot v$ denotes the standard scalar product of two vectors, regardless of the underlying dimension.} $s_n=\min_{|u|=1}(\text{Cov}(S_n)u\cdot u)$. Then, in the vector-valued case, we will show that this assumption holds true for several classes of inhomogeneous contracting Markov chains. 

The proof of the results relies on a recent modification of \cite[Theorem 1.3]{GO}, together with a block-partition argument, which in some sense reduces the problem to the case when the variance or the covariance matrix of $S_n$ grows linearly fast in $n$. More precisely, we show that there are ``intervals"   $I_{j}=\{a_j,a_{j}+1,...,b_j\}$  in the positive integers  so that  $a_1=1$ and $b_j+1=a_j$ (i.e. $\bbN=\cup_{j}I_j$) and the variance (covariance matrix) of each partial sum of the form $\sum_{j=1}^{k}\Xi_{j}$, $\Xi_j=\sum_{s\in I_j}X_s$ grows linearly fast in $k$. In this paper the sets $I_j$ will be referred to as ``blocks". Once the blocks $I_j$ are constructed  the proof of the ASIP for $S_n$ has two steps: first,  we prove the ASIP for the sequence $\tilde S_k=\sum_{j=1}^{k}\Xi_{j}$ using the modification of \cite[Theorem 1.3]{GO} and then we approximate $S_n$ by $\tilde S_{k_n}$, where $k_n$ is the largest index so that $I_{k_n}\subset\{1,2,...,n\}$, and show that $k_n\asymp s_n=\min_{|u|=1}(\text{Cov}(S_n)u\cdot u)$.

\section{Preliminaries and main results}

Let $X_1,X_2,...$ be a sequence of zero-mean uniformly bounded $d$-dimensional random vectors defined on a  probability space $(\Om,\cF,\bbP)$. For each $j\in\bbN$, let $\cF_j$ denote the $\sig$-algebra generated by $X_1,...,X_j$ and let $\cF_{j,\infty}$  denote the $\sig$-algebra generated by $X_k$ for $k\geq j$.
Recall  that the $\al$ and $\phi$ mixing coefficients of the sequence are given by 
\begin{equation}\label{al def}
\al(k)=\sup\left\{\left|\bbP(A\cap B)-\bbP(A)\bbP(B)\right|: A\in\cF_j,\, B\in\cF_{j+k,\infty},\, j\in\bbN\right\}
\end{equation} 
and 
\begin{equation}\label{phi def}
\phi(k)=\sup\left\{\left|\bbP(B|A)-\bbP(B)\right|: A\in\cF_j,\, B\in\cF_{j+k,\infty},\, j\in\bbN,\,\,\bbP(A)>0\right\}.
\end{equation} 
Then both $\al(\cdot)$ and $\phi(\cdot)$ measure the long range dependence of the sequence $\{X_j\}$ in the sense that $X_j$'s are independent if and only if both sequences $\al(\cdot)$ and $\phi(\cdot)$ are identically $0$.

We will assume here that there are constants $C>0$, $\del\in(0,1)$ and $n_0\in\bbN$ so that 
\begin{equation}\label{al mix}
\al(n)\leq 	C\del^n,\,\, \text{ for all }n\in\bbN
\end{equation}
and 
\begin{equation}\label{phi half}
\phi(n_0)<\frac12.
\end{equation}
These are the mixing (weak-dependence) assumptions discussed in Section \ref{Intro}.
\begin{remark}
It is clear from the definitions of $\al(k)$ and $\phi(k)$ that $\al(k)\leq \phi(k)$. Hence, both conditions \eqref{al mix} and \eqref{phi half} are in force when $\phi(n)\leq C\del^n$ for some $C>0$ and $\del\in(0,1)$.  Note also that for Markov chains, condition \eqref{phi half} already implies that  $\phi(n)$ decays exponentially fast to $0$, and so in this case \eqref{phi half}  implies \eqref{al mix}. In any case, all the result in this paper are new even when $\phi(n)$ decays exponentially fast\footnote{In fact, this was the main mixing assumption in a previous version of this paper https://arxiv.org/abs/2005.02915v3}.
\end{remark}

Next, for each $n\in\mathbb N$ set 
\[
S_n=\sum_{k=1}^n X_k
\]
and put $V_n=\text{Cov}(S_n)$ (which is a $d\times d$ matrix). 
For all $n,m\in\bbN$ so that $n\leq m$   set 
$$
S_{n,m}=\sum_{j=n}^{m}X_j,\,\,V_{n,m}=\text{Cov}(S_{n,m}),\,s_n=\min_{|u|=1}(V_n u\cdot u)
$$
where  $|u|$  denotes the Euclidean norm of a vector $u\in\bbR^d$  and
$u\cdot v$ denotes the standard scalar product of two vectors $u,v\in \bbR^d$. 
Then in the scalar case $d=1$ we have $s_n=V_n=\text{Var}(S_n)$.

Next,  for a random variable $Z:\Om\to\bbR^d$ and a number $p\in[1,\infty)$ let us denote $\|Z\|_{L^p}=\left(\int|Z(\om)|^p d\bbP(\om)\right)^{1/p}$.
We consider here the following condition.

\begin{assumption}\label{AssVV}
There are constants $C_1,C_2\geq1$ with the following property:  for every  pair of positive integers $n$ and $m$ so that $n\leq m$ and $\|S_{n,m}\|_{L^2}\geq C_1$ we have
$$
\max_{|u|=1}(V_{n,m} u\cdot u)\leq C_2\min_{|u|=1}(V_{n,m} u\cdot u).
$$
\end{assumption}
This assumption trivially holds true for real-valued sequences, and in Section \ref{SecVV} we will verify it for certain classes of additive vector-valued functionals $X_j=f_j(\xi_j)$ of inhomogeneous ``sufficiently contracting" Markov chains $\{\xi_j\}$. Note also that 
$$
V_{n,m} u\cdot u=\text{Var}(S_{n,m}\cdot u)
$$
and so Assumption \ref{AssVV} gives us a certain type of uniform control over these variances\footnote{However, $s_n$ can still grow arbitrarily slow.}.

%
%

Our main result here is the following:
\begin{theorem}\label{Main Thm}
Under Assumption \ref{AssVV} we have the following.
Suppose  that (\ref{al mix}) and \eqref{phi half} hold true and that $\lim_{n\to\infty}s_n=\infty$. 
Then for every $\ve>0$ there is a coupling between $X_1,X_2,...$ and  a sequence of  independent zero-mean Guassian random vectors $Z_1,Z_2,\ldots$ so that
\begin{equation}\label{Rate}
\left|S_n-\sum_{j=1}^{n}Z_j\right|=o(s_n^{1/4+\ve}),\,\,\text{almost surely.}
\end{equation}
 Moreover, there is a constant $C=C_\ve>0$ so that for all $n\geq1$  and a unit vector $u\in\mathbb R^d$,
\begin{equation}\label{Var est1}
\left\|S_n\cdot u\right\|_{L^2}^2-Cs_n^{1/2+\ve}\leq \left\|\sum_{j=1}^n Z_j\cdot u\right\|_{L^2}^2\leq \left\|S_n\cdot u \right\|_{L^2}^2+Cs_n^{1/2+\ve}.
\end{equation}
\end{theorem}

\begin{remark}
\,

\textit{(i)} In the scalar case $d=1$,  \eqref{Var est1} yields that the difference between the variances is $O(V_n^{\frac 12+\delta})$. Thus, using \eqref{Var est1} together with~\cite[Theorem 3.2 A]{HR}, we conclude that  in the scalar case, for every $\ve>0$ there is a coupling of $\{X_n\}$ with a standard Brownian motion $\{W_t:\,t\geq0\}$ so that
\begin{equation}\label{ZZZ}
\left|\sum_{j=1}^{n}X_j-W_{V_n}\right|=o(V_n^{\frac14+\ve}),\quad\text{a.s.}
\end{equation}
A corresponding result in the vector-valued case seems less plausible because in the non-stationary setup the structure of the covariance matrix $V_n$ does not stabilize as $n\to\infty$, which makes it less likely that we can approximate $S_n$ by a single Gaussian process like a  standard $d$-dimensional Brownian motion.
\vskip0.1cm

\textit{(ii)} For stationary sequences $\{X_n\}$, it was shown in \cite[Theorem 1.4]{Shao} that if $\phi(n)\ll \ln^{-r} n$  and $\bbE[|X_n|^{2+\del}]<\infty$
for some $\del>0$ and $r>(2+\del)/(2+2\del)$, then
there is a coupling  of $\{X_n\}$ with a standard Brownian motion so that the left hand side of (\ref{Rate}) is of order $o(V_n^{1/2}\ln^{-\te} V_n)$
for an arbitrary $0<\te<(r(1+\del))/(2(2+2\del))-\frac14$. In comparison with \cite{Shao}, we get better ASIP rates in the non-stationary case, but only for uniformly bounded exponentially  fast $\al$-mixing sequences such that $\lim_{n\to\infty}\phi(n)<\frac12$.
\vskip0.1cm

\textit{(iii)} We would like to stress that even in the scalar case $d=1$ no growth rates on the variance (such as $V_n\geq n^\ve$) are required in Theorem \ref{Main Thm}. This is in contrast, for instance, with \cite{Hyd} where it was assumed that $V_n\geq n^{\frac12+\del}$, and \cite{GO} and \cite{WZ} where a linear growth was assumed. Note that in the latter papers vector-valued variables were considered.
\vskip0.1cm

\textit{(iv)} Many papers about the ASIP rely on martingale approximation (e.g. \cite{Hyd} and \cite{WZ}). However, to the best of our knowledge, the best rate in the vector-valued case that can be achieved using martingales (in the stationary case) is $o(n^{1/3}(\log n)^{1+\ve})=o\big(s_n^{1/3}(\log s_n)^{1+\ve}\big)$ (see \cite{CDF}), and so an attempt to use existing results for martingales  seems to  yield  weaker rates than the ones obtained in Theorem \ref{Main Thm}.

\end{remark}

\section{A linearization of the growth rate of the covariance matrix}\label{Sec3}
The main step in the proof of Theorem \ref{Main Thm} is to make a certain reduction to the case when $s_n=\min_{|u|=1}(V_n u\cdot u)$ grows linearly fast in $n$. This is the content of the following result.

\begin{proposition}\label{VarPrp}
Suppose that\footnote{Note that this series converges when \eqref{al mix} holds true.} 
$\sum_{m=1}^\infty\left(\al(m)\right)^{1-2/p}<\infty$ for some $p>2$ and  that $\lim_{n\to\infty}s_n=\infty$. Then there  are constants $A_1,A_2>0$ and disjoint sets $I_j=\{a_j,a_j+1,...,b_j\}\subset\bbN$ whose union cover $\bbN$ (so that $a_1=1$ and $a_{j+1}=b_j+1$ for all $j$) and for all $j\in\bbN$ and a unit vector $u$ we have
\begin{equation}\label{A}
A_1\leq \left\|\sum_{k\in I_j}X_k\cdot u \right\|_{L^2}\leq \max_{m\in I_j}\left\|\sum_{k=a_j}^{m}X_k\cdot u \right\|_{L^2}\leq A_2.
\end{equation}
and so
\begin{equation}\label{Block Cntrl}
\sup_{j\in\bbN}\, \max_{m\in I_j}\left\|\sum_{k=a_j}^{m}X_k\right\|_{L^2}\leq A_2.
\end{equation}
Moreover, let $k_n=\max\{k: b_k\leq n\}$ and set $\Xi_j=\sum_{k\in I_j}X_k$. Then the following statement hold true.

(i) There are  constants $R_1,R_2>0$ so that  for every  $n$ large enough and all unit vectors $u$,
\begin{equation}\label{k n vn.1}
R_1k_n\leq \text{Var}(S_n\cdot u)=\text{Cov}(S_n)u\cdot u\leq R_2k_n.
\end{equation}

(ii) If also  \eqref{phi half} is valid, then for  every $\ve>0$ we have
\begin{equation}\label{Last}
\left|S_n-\sum_{j=1}^{k_n}\Xi_j\right|=o(s_n^{\ve}),\,\,\bbP-\text{a.s.}
\end{equation}
\end{proposition}

\subsection*{Proof of  Proposition \ref{VarPrp}}
First, let us fix some unit vector $u_0$, and set $\xi_j=X_j\cdot u_0$. 
For every finite  $M\subset\bbN$ set
\[
S(M)=\sum_{j\in M}X_j\cdot u_0=\sum_{j\in M} \xi_j.
\]

Next, let $A>1$ and $r\in\bbN$ be  sufficiently large constants which are yet to be determined. 
Let us construct a sequence $M_j,\, j\in\bbN$ of intervals (blocks) in the positive integers as follows.
Let $p_1$ be the first index $p$ so that $\|\sum_{j=1}^{p}\xi_j\|_{L^2}\geq \sqrt{A}$ and set $M_1=\{1,2,...,p_1\}$. Next, given that $M_j=\{q_j,q_{j}+1,...,p_j\}$ was constructed we define $q_{j+1}=p_j+r$ and $M_{j+1}=\{q_{j+1}, q_{j+1}+1,...,p_{j+1}\}$, where $p_{j+1}$ is the first index $p\geq q_{j+1}$ so that $\|S(\{q_{j+1},...,p\})\|_{L^2}\geq \sqrt A$.  Then the blocks $M_j=\{q_j,q_j+1,...,p_j\}$ satisfy the following properties:
\begin{enumerate}
\item $M_1$ contains $1$ and  for each $j$ the block $M_j$ is to the left of $M_{j+1}$, and $\min M_{j+1}-\max M_j=r$;
\vskip0.1cm
\item For each $j$ we have 
$\sqrt A\leq \|S(M_j)\|_{L^2}\leq\sqrt A+L,$\,\,\,$L=\sup_n(\text{ess-sup}|X_n|)$ and 
\begin{equation}\label{Mx 2}
\max_{s\in M_j,\, s<p_j} \|S(\{q_j,q_{j}+1,...,s\})\|_{L^2}<\sqrt A\leq \|S(M_j)\|_{L^2}.
\end{equation}

\end{enumerate}
Next,  let us define $I_j=M_j+\{0,1,...,r-1\}$.
Then the block $I_j$ is to the left of $I_{j+1}$ and the union of the $I_j$'s cover $\bbN$. Thus we can write $I_{j}=\{a_j,a_{j}+1,...,b_j\}$ with $a_{j+1}=b_j+1$ and $a_1=1$. 

We will break down the rest of the proof of Proposition \ref{VarPrp} into a few steps. 
Between the steps we will introduce appropriate restrictions on  $r$ and $A$, and the sets $I_j$ corresponding to appropriate choices of $r$ and $A$ will satisfy all the properties described in Proposition \ref{VarPrp}.

The first result we need is the following:
\begin{lemma}\label{L1}
For every $p>2$ there is a constant $C_p\geq 1$ which does not depend on $A$ or $r$ so that for every $1\leq i<j$ we have
\begin{equation}\label{p cov}
\left|\text{Cov}(S(M_i),S(M_j))\right|\leq C_p\|S(M_i)\|_{L^2}\|S(M_j)\|_{L^2}\left(\al(r(j-i))\right)^{1-2/p}.
\end{equation}
\end{lemma}
\begin{proof}
By applying \cite[Corollary A.2]{Hall} we get that
\begin{equation}\label{p cov0}
\left|\text{Cov}(S(M_i),S(M_j))\right|\leq 8\|S(M_i)\|_{L^p}\|S(M_j)\|_{L^p}\left(\al(r(j-i))\right)^{1-2/p}.
\end{equation}
On the other hand, since \eqref{phi half} holds, by applying \cite[Theorem 6.17]{PelBook}, taking into account that $X_j$ are uniformly bounded and using \eqref{Mx 2} we get that 
\begin{equation}\label{M p}
\|S(M_i)\|_{L^p}\leq A_p(1+\|S(M_i)\|_{L^2})
\end{equation}
where $A_p\geq 1$ is a constant that depends only on $p$, $n_0$ from \eqref{phi half} and $\ve=\frac12-\phi(n_0)$. Now the proof is completed by recalling that $\|S(M_i)\|_{L^2}\geq\sqrt A\geq1$ (and so we can take $C_p=32A_p$).
\end{proof}

Next, let $p$ be as in Proposition \ref{VarPrp}.  Since $\sum_{m=1}^\infty\left(\al(m)\right)^{1-2/p}<\infty$ there exists  $r_0\in\bbN$
so that\footnote{Indeed $\sum_{m=1}^{\infty}\left(\al(rm)\right)^{1-2/p}\leq \sum_{m=r}^\infty\left(\al(m)\right)^{1-2/p}\to 0\text{ as }r\to\infty$.} 
\begin{equation}\label{r 0}
4 C_p\sum_{m=1}^{\infty}\left(\al(r_0m)\right)^{1-2/p}\leq 1
\end{equation}
 where $C_p$ is the constant from Lemma \ref{L1}. Henceforth we will set $r=r_0$.
 
The second result we need is as follows.

\begin{lemma}\label{Lemma 2}
If the sets $\{M_j\}$ are constructed  with $r=r_0$ so that  \eqref{r 0} holds true,  then for every $k\in\bbN$  we have
\[
\frac12\sum_{i=1}^{k}\text{Var}(S(M_i))\leq \text{Var}(S(M_1\cup M_2\cup\cdots\cup M_{k}))\leq \frac{3}2\sum_{i=1}^{k}\text{Var}(S(M_i)).
\] 
\end{lemma}

\begin{proof}
First,
$$
 \text{Var}(S(M_1\cup M_2\cup\cdots\cup M_{k}))=\sum_{i=1}^{k}\|S(M_i)\|_{L^2}^2+2\sum_{1\leq i<j\leq k}\text{Cov}(S(M_i),S(M_j)).
$$
Next, set $\gamma(k)=\big(\al(k)\big)^{1-2/p}$.
Then by \eqref{p cov}, 
\begin{equation}\label{SimTo}
2\sum_{1\leq i<j\leq k}|\text{Cov}(S(M_i),S(M_j))|\leq 2C_p\sum_{1\leq i<j\leq k}\gamma(r(j-i))\|S(M_i)\|_{L^2}
\|S(M_j)\|_{L^2}
\end{equation}
$$\leq  C_p\sum_{1\leq i<j\leq k}\gamma(r(j-i))(\|S(M_i)\|_{L^2}^2+\|S(M_j)\|_{L^2}^2)
= C_p\sum_{j=2}^{k}\|S(M_j)\|_{L^2}^2\sum_{i=1}^{j-1}\gamma(r(j-i))+$$ 
$$C_p\sum_{i=1}^{k-1}\|S(M_i)\|_{L^2}^2\sum_{j=i+1}^{k}\gamma(r(j-i))\leq \left(2C_p\sum_{m\geq 1}\gamma(rm)\right)\sum_{j=1}^{k}\|S(M_j)\|_{L^2}^2.
$$
The proof is completed using that $2C_p\sum_{m\geq 1}\gamma(rm)\leq\frac12$.  
\end{proof}

Next, let $r_0$ satisfy \eqref{r 0} and  set
$Q_0=2C_pr_0d^2L^2\sum_{m\geq1}\left(\al(m)\right)^{1-2/p}+(r_0dL)^2$, where $d$ is the dimension of the random vectors $X_j$. For each $A$ set
$$
Q(A)=Q(A,r_0,p,L)=Q_0+2\sqrt{3AQ_0}.
$$
Then   $Q(A)/A\to 0$ as $A\to\infty$.  Let $A_0>1$ be so that for all $A\geq A_0$ we have  
$$
\sqrt{A}\geq 2r_0dL,\,\,  A\geq 4Q(A)\,\, \text{ and }\,\, (\sqrt A+L)^2\leq 2A.
$$  
Note that the second restriction on $A$ guarantees that  $A\leq \text{Var}(S(M_j))\leq 2A$ for each $j$.

The last auxiliary result we need before completing the proof of Proposition \ref{VarPrp} is as follows.

\begin{lemma}\label{Cor1}
Suppose that the sets $M_j$ are constructed with $r=r_0$ so that \eqref{r 0} holds true and with $A\geq A_0$. 
Fix some  $k\in\bbN$ and set $\Lambda_1=M_{1}\cup M_{2}\cup\cdots\cup M_{k}$ and $\Lambda_2=I_{1}\cup I_{2}\cup\cdots\cup I_{k}$. Then,
\begin{equation}\label{CorEq}
\left|\frac{\text{Var}(S(\Lambda_2))}{\text{Var}(S(\Lambda_1))}-1\right|\leq \frac{2Q(A)}{A}\leq \frac12.
\end{equation}

\end{lemma}
\begin{proof}
Let $X=S(\Lambda_1)$ and $Y=S(\Lambda_2)-X$. Then 
$$
\text{Var}(X+Y)=\text{Var}(X)+\text{Var}(Y)+2\text{Cov}(X,Y)
$$
and so by the Cauchy-Schwarz inequality, 
\begin{equation}\label{One}
\left|\text{Var}(X+Y)-\text{Var}(X)\right|\leq \text{Var}(Y)+2\left(\text{Var}(X)\text{Var}(Y)\right)^{1/2}.
\end{equation}
Now,by Lemma \ref{Lemma 2},
\begin{equation}\label{Two}
\frac{Ak}{2}\leq \frac12\sum_{j=1}^{k}\text{Var}(S(M_j))\leq \text{Var}(X)\leq \frac 32\sum_{j=1}^{k}\text{Var}(S(M_j))\leq 3Ak
\end{equation}
where we have used that $A\leq \text{Var}(S(M_j))\leq 2A$.
On the other hand, let $D_j=I_j\setminus M_j$. Then $Y=\sum_{j=1}^k S(D_j)$ and so
$$
\text{Var}(Y)=\text{Cov}(Y,Y)\leq\sum_{j=1}^{k}|\text{Cov}(S(D_j),Y)|.
$$
Now, fix some $j$ and write  $D_j=\{d_j+1,...,d_j+r-1\}$. Then 
$$
|\text{Cov}(S(D_j),Y)|\leq\sum_{m\leq d_j}|\text{Cov}(S(D_j),X_m)|+\sum_{m\geq d_j+r}|\text{Cov}(S(D_j),X_m)|+\text{Var}(S(D_j)).
$$
Next, by applying \cite[Corollary A.2]{Hall} and using \eqref{M p} we see that if $m\not\in D_j$ then
$$
|\text{Cov}(S(D_j),X_m)|\leq C_p\|S(D_j)\|_{L^p}\|X_m\|_{L^p}\left(\al(\rho_{m,j})\right)^{1-2/p},\,\,\,\rho_{m,j}=\min_{s\in D_j}|m-s|.
$$
Using also that $\|S(D_j)\|_{L^p}\leq rdL$ and $\|X_m\|_{L^p}\leq dL$ for every $p>1$ we see that 
$$
|\text{Cov}(S(D_j),Y)|\leq 2C_p(rdL)(dL)\sum_{m\geq1}\left(\al(m)\right)^{1-2/p}+(rdL)^2=Q_0.
$$
Thus, 
$$
\text{Var}(Y)\leq Q_0k.
$$
Finally, using \eqref{One} and \eqref{Two} we conclude that 
$$
\left|\text{Var}(X+Y)-\text{Var}(X)\right|\leq \left(Q_0+2\sqrt{3AQ_0}\right)k=Q(A)k.
$$
The proof is completed by
dividing the above left hand side by $\text{Var}(X)$ and using \eqref{Two}.
\end{proof}

\begin{proof}[Completion of the proof of Proposition \ref{VarPrp}]
Let us construct the blocks $\{I_j\}$ with  constants $A\geq A_0$ and $r=r_0$ with the same restrictions described before.
First, since $\sqrt{A}\geq 2r_0 dL$,  using the second property of $M_j$ and that $I_j\setminus M_j$ is of cardinality $r_0-1$  we obtain \eqref{A}  with the specific unit vector $u=u_0$ and the constants $A_1=\frac12\sqrt{A}$ and $A_2=\frac32\sqrt{A}$. By using Assumption \ref{AssVV}, we see that if $A$ is large enough then \eqref{A}  holds true all unit vectors $u$, possibly with different constants. The estimate \eqref{Block Cntrl} follows by taking the supremum over all unit vectors $u$ in the third inequality from the left in \eqref{A}. 
Next, by applying Lemmas \ref{Lemma 2} and \ref{Cor1}, we see that  \eqref{k n vn.1} holds true  with the specific unit vector $u=u_0$.
Thus, by Assumption \ref{AssVV}, if $A$ is large enough then  \eqref{k n vn.1} holds  for an arbitrary unit vector  (possibly with different constants).

In order to prove \eqref{Last}, let us assume \eqref{phi half}. 
For each $q\geq1$ 
set 
\[
\cD_q:=\max_{b_{q}<n\leq b_{q+1}}|S_n-S_{b_q}|=\max_{m\in I_{q+1}}\left|\sum_{j=a_{q+1}}^m X_j\right|
\]
where in the second inequality we used that $b_{q}+1=a_{q+1}$.
Then with $\Xi_j=\sum_{k\in I_j}X_k$ and  $k_n=\max\{k: b_k\leq n\}$ we have
\begin{equation}\label{Fin}
\left|S_n-\sum_{j=1}^{k_n}\Xi_j\right|\leq \cD_{k_n}.
\end{equation}
By applying \cite[Theorem 6.17]{PelBook} with the random variables $\{X_{n}: n\in I_{q+1}\}$ (which is possible due to \eqref{phi half})
 we see  that for every $p>2$ there are constants $c_{p}$ and $R_p$ so that for all  $q\in\bbN$ we have 
$$
\|\cD_q\|_{L^p}\leq R_p\left(\left\|\max\{|X_n|:n\in I_{q+1}\}\right\|_{L^p}+\max\{\|S_n-S_{b_q}\|_{L^2}: n\in I_{q+1}\}\right)
\leq c_p
$$
where in the second inequality we have used that $\sup_n(\text{ess-sup}|X_n|)<\infty$ and  \eqref{Block Cntrl}. 
Thus, by applying the Markov inequality we see that for every $\ve>0$ and $p>2$ we have
\[
P(|\cD_q|\geq q^\ve)=P(|\cD_q|^p\geq q^{\ve p})\leq c_p^pq^{-\ve p}.
\]
Taking $p>1/\ve$ we get from the Borel-Cantelli lemma that 
\begin{equation}\label{abo}
|\cD_q|=O(q^\ve),\,\text{a.s.}
\end{equation}
The desired estimate (\ref{Last}) follows  by plugging in $q=k_n$ in \eqref{abo}  and using (\ref{Fin}) and (\ref{k n vn.1}).
\end{proof}

\section{ASIP: proof Theorem \ref{Main Thm}}

The proof of Theorem \ref{Main Thm} is based  on an  application of  \cite[Theorem 2.1]{DH} with an arbitrary $p>4$. The latter theorem is a modification of  \cite[Theorem 1.3]{GO} suited for more general non-stationary sequences of random vectors. The standing assumption in both theorems can be described as follows.
Let $(A_1, A_2, \ldots )$ be an $\bbR^d$-valued process on some probability space $(\Omega, \mathcal F, \mathbb P)$. Then there exists $\varepsilon_0>0$ and $C,c>0$ such that for all $n,m\in \bbN$, $a_1<a_2< \ldots <a_{n+m+k}$, $k\in \bbN$ and $t_1,\ldots ,t_{n+m}\in\mathbb R^d$ with $|t_j|\leq\varepsilon_0$, we have that
\begin{eqnarray}\label{(H)}
\Big|\mathbb E\big(e^{i\sum_{j=1}^nt_j\cdot(\sum_{\ell=a_j}^{a_{j+1}-1}A_\ell)+i\sum_{j=n+1}^{n+m}t_j\cdot(\sum_{\ell=a_j+k}^{a_{j+1}+k-1}A_\ell)}\big)\\
-\mathbb E\big(e^{i\sum_{j=1}^nt_j\cdot(\sum_{\ell=a_j}^{a_{j+1}-1}A_\ell)}\big)\cdot\mathbb E\big(e^{i\sum_{j=n+1}^{n+m}t_j\cdot(\sum_{\ell=a_j+k}^{a_{j+1}+k-1}A_\ell)}\big)\Big|\nonumber\leq C(1+\max|a_{j+1}-a_j|)^{C(n+m)}e^{-ck}.\nonumber
\end{eqnarray}

The first part of the proof is to show that $A_j=\Xi_j=\sum_{k\in I_j}X_k$ satisfies \eqref{(H)}, which follows directly from the exponential $\al$-mixing rates \eqref{al mix}. Next, let us verify the rest of the conditions of  \cite[Theorem 2.1]{DH}. Set
$$\cA_n=\sum_{j=1}^n A_j.$$
Then, by applying \eqref{k n vn.1} with $b_{n}$ instead of $n$ we see that for all $n$ large enough we have
$$
\min_{|u|=1}\left(\text{Cov}(\cA_n)u\cdot u\right)\geq Cn
$$
where $C>0$ is a constant.  This shows that the first additional condition in \cite[Theorem 2.1]{DH} is satisfied. To show that $A_j$ are uniformly bounded in $L^p$, combining our assumption \eqref{phi half} with \cite[Theorem 6.17]{PelBook}  and taking into account  \eqref{Block Cntrl}, we see that for every $p>2$,
\begin{equation}\label{BP}
B_p:=\sup_j\|A_j\|_{L^p}<\infty.
\end{equation}
The last condition we need to verify is that
\begin{equation}\label{Need Last}
\left|\text{Cov}(A_n\cdot u, A_{n+k}\cdot u)\right|\leq C_0\eta^k
\end{equation}
for  some $C_0>0$, $\eta\in(0,1)$,  all $k,n\in\bbN$ and all unit vectors $u\in\bbR^d$. To establish that, let us fix some $p>2$. Then by \cite[Corollary A.2]{Hall} we have
$$
\left|\text{Cov}(A_n\cdot u, A_{n+k}\cdot u)\right|\leq \|A_n\cdot u\|_{L^p}\|A_{n+k}\cdot u\|_{L^p}\left(\al(k)\right)^{1-2/p}$$
and so by \eqref{al mix} and \eqref{BP}  we see that  \eqref{Need Last} holds true with $C_0=B_p^2C^{1-2/p}$ and $\eta=\del^{1-2/p}$ (where $C$ and $\del$ come from \eqref{al mix}).

Next, by applying \cite[Theorem 2.1]{DH} with the sequence $A_j=\Xi_j=\sum_{k\in I_j}X_k$ we conclude that 
 there is a coupling between the sequence $A_1,A_2,...$ and a sequence $Z_1,Z_2,...$ of independent centered Gaussian random vectors so that for every $\ve>0$,
\begin{equation}\label{asip.0}
\left|\sum_{i=1}^{k}A_i-\sum_{j=1}^k Z_j\right|=o(k^{\frac 14+\ve}),\,\,\text{a.s.}
\end{equation}
and  all the properties specified in Theorem \ref{Main Thm} hold true for the new sequence $A_j=\Xi_j$.
Now Theorem \ref{Main Thm} follows by plugging in $k=k_n$ in \eqref{asip.0}, using \eqref{k n vn.1}, and then  approximating $S_n$ by $\cA_{k_n}=\sum_{j=1}^{k_n}\Xi_j$, relying on  \eqref{Last} and using the, so-called, Berkes-Philipp lemma (which allows us to further couple $(X_j)$ with the Gaussian sequence).

\section{Verification of the additional conditions in the non-scalar case: Markov chains}\label{SecVV}
Assumption \ref{AssVV} trivially holds true for real-valued random variables $X_j$.
In this section we discuss natural sufficient conditions for Assumption \ref{AssVV} for certain additive functionals of contracting  Markov chains.

\subsection*{Dobrushin's contracting chains}
Let us recall the definition of Dobrushin's contraction coefficients $\pi(\cdot)$ (see \cite{Dub}). If $Q(x,\cdot)$ is a regular family of Markov transition operators between two spaces $\mathcal X$ and $\mathcal Y$,
then 
\[
\pi(Q)=\sup\{|Q(x_1,E)-Q(x_2,E)|:\,x_1,x_2\in\cX, E\in\cB(\cY)\}
\]
where $\cB(\mathcal Y)$ is the underlying $\sigma$-algebra on $\mathcal Y$.

Let  $\{\xi_j\}$ be a Markov chain with corresponding state spaces $\cX_j$. Let $Q_j(x,\Gamma)=\bbP(\xi_{j+1}\in\Gamma|\xi_j=x)$ and suppose that 
\begin{equation}\label{Cont}
\del:=\sup_j\pi(Q_j)<1.
\end{equation}
Then, as proven in \cite{VarSeth}, the chain $\{\xi_j\}$ is exponentially fast $\phi$-mixing.
Let us  take a sequence $f_j$ of bounded measurable functions on $\cX_j$ and set $X_j=f_j(\xi_j)-\bbE[f_j(\xi_j)]$.
Then by the results\footnote{In \cite{VarSeth} only the lower bound was derived, however in this setup the upper bound is easier to obtain.} in \cite{VarSeth} (see also \cite[Proposition 13]{Pel}), there are positive constants $A=A_\del$ and $B=B_\del$ so that for every $n,m$ with $n\leq m$ and each unit vector $u$,
$$
A\sum_{j=n}^m\text{Var}(X_j\cdot u)\leq \text{Var}(S_{n,m}\cdot u)\leq B\sum_{j=n}^m\text{Var}(X_j\cdot u).
$$
We thus get the following result.
\begin{proposition}\label{Dob}
Assumption \ref{AssVV} (and hence Theorem \ref{Main Thm}) holds true if $\del<1$ and there is a constant $C\geq 1$ so that for every $j\in\bbN$ we have
$$
\max_{|u|=1}(\text{Cov}(X_j)u\cdot u)\leq C\min_{|u|=1}(\text{Cov}(X_j) u\cdot u).
$$
\end{proposition}

\subsubsection{Uniformly elliptic chains}
In this section we consider a (somewhat) less general class of Markov chains $\{\xi_j\}$, but more general functionals. 
 Let $\{\xi_j\}$ be a Markov chain with transition densities
\[
\bbP(\xi_{j+1}\in\Gamma|\xi_{j}=x)=\int_{\Gamma}p_j(x,y)d\mu_{j+1}(y)
\] 
where $\mu_{j+1}$ is a measure on the state space $\mathcal X_{j+1}$ of $\xi_{j+1}$ and   $\Gamma\subset\mathcal X_{j+1}$ is a measurable set.
We assume that there exists $\varepsilon_0>0$ so that for any $i$ we have $\sup_{x,y}p_i(x,y)\leq 1/\varepsilon_0$, and the second step transition densities of $\xi_{i+2}$ given $\xi_i$ are bounded  below by $\varepsilon_0$ (this is the uniform ellipticity condition): 
\[
\inf_{i\geq1}\inf_{x,z}\int p_i(x,y)p_{i+1}(y,z)d\mu_{i+1}(y)\geq \varepsilon_0.
\]
Then the resulting Markov chain $\{\xi_j\}$ is exponentially fast $\phi$-mixing (see \cite[Proposition 1.22]{DS}). Note that if the first step transition densities $p_i$ were bounded below then we would get \eqref{Cont}, but the assumption about the second step transition densities does necessary yield \eqref{Cont}.

Next, we take  a uniformly bounded sequence of measurable functions $f_j:\mathcal X_j\times\mathcal X_{j+1}\to\mathbb R^d$ and set $X_j=f_j(\xi_j,\xi_{j+1})-\bbE[f_j(\xi_j,\xi_{j+1})]$. Let us fix some unit vector $u$.
Then, by applying
 \cite[Theorem 2.1]{DS}  with the real-valued functions $f_j\cdot u$ (which are uniformly bounded in both $j$ and $u$) we see that there are non-negative numbers $u_i(f;u)=u_i(f_{i-2}\cdot u,f_{i-1}\cdot u,f_i\cdot u)$ and constants $A,B,C,D>0$ which depend only on $\varepsilon_0$ and $K:=\sup_j\sup|f_j|$
 so that  for all $m,n$ with  $m-n\geq 3$ we have
\begin{equation}\label{Var2}
A\sum_{j=n+3}^{m} u_j^2(f;u)-B\leq\text{Var}(S_{n,m}\cdot u)\leq C\sum_{j=n+3}^{m} u_j^2(f;u)+D
\end{equation}
where  we recall that $S_{n,m}=\sum_{j=n}^{m}X_j$.
The numbers $u_i(f;u)$ are given in  \cite[Definition 1.14]{DS}: $u_i^2(f;u)=(u_i(f;u))^2$ is the variance of the balance (in the terminology of \cite{DS}) function $\Gamma_i=\Gamma_{i,f\cdot u}$ given by 
\begin{eqnarray*}
\Gam_i(x_{i-2},x_{i-1},x_i,y_{i-1},y_{i},y_{i+1})=f_{i-2}(x_{i-2},x_{i-1})\cdot u+f_{i-1}(x_{i-1},x_i)\cdot u+f_{i}(x_i,y_{i+1})\cdot u\\-f_{i-2}(x_{i-2},y_{i-1})\cdot u-f_{i-1}(y_{i-1},y_i)\cdot u-f_{i}(y_i,y_{i+1})\cdot u
\end{eqnarray*}
corresponding to the hexagon generated by $(x_{i-1},x_{i},x_{i+1};y_{i-1},y_{i},y_{i+1})$, with respect to the probability measure on the space of hexagons positioned at ``time" $i$, as  introduced in \cite[Section 1.3]{DS}.
We thus have the following result.
\begin{proposition}\label{UE}
Assumption \ref{AssVV} (and hence Theorem \ref{Main Thm}) holds true if  there is a constant $C\geq1$ so that for each $j$ the matrix $B_j$ defined by $(B_j)_{k,\ell}=\frac12\big(u_j^2(f,e_k)+u_j^2(f,e_\ell)\big)$ (where $e_m$ is the $m$-th standard unit vector), satisfies
$$
\max_{|u|=1}(B_j u\cdot u)\leq C\min_{|u|=1}(B_j u\cdot u).
$$
\end{proposition}

\subsection*{Weaker results for uniformly contracting Markov chains}
Let $\{\xi_j\}$ be a Markov chain. Let us consider the transition operators $Q_j$ given by $Q_jg(x)=\mathbb E[g(\xi_{j+1})|\xi_j=x]$.
For each $j\geq1$
let $\rho_j$ be the $L^2$-operator norm of the restriction of $Q_j$ to the space of zero-mean square-integrable functions $g(\xi_{i+1})$ (see \cite{Pel}). We assume here that  
\[
\rho:=\sup_{j}\rho_j<1.
\]
In these circumstances the Markov chain $\{\xi_j\}$ is exponentially fast $\rho$-mixing (see \cite{Pel}), and so by \cite[(1.22)]{BrMix} we get \eqref{al mix}.
Note also that by \cite[Lemma 4.1]{VarSeth} we have,
\[
\rho_j\leq\sqrt{\pi(Q_j)}
\]
and so this is a weaker assumption than \eqref{Cont}

Let $f_j:\mathcal X_j\to\bbR^d$ be a sequence of measurable uniformly bounded functions and set $X_j=f_j(\xi_j)$. We prove here the following result.

\begin{theorem}\label{Thm1}
Suppose that $s_n=\min_{|u|=1}(V_n u\cdot u)\geq c_0n^{\delta_0}$ for some constants $c_0,\delta_0>0$. Assume also that 
 there exists $C\geq 1$ so that for each $j$ we have
\begin{equation}\label{Ut}
\max_{|u|=1}(\text{Cov}(X_j)u\cdot u)\leq C\min_{|u|=1}(\text{Cov}(X_j) u\cdot u).
\end{equation}
Then
there is a coupling of $X_1,X_2,...$ with a sequence of independent centered Gaussian vectors $Z_1,Z_2,...$ with the properties described in Theorem \ref{Main Thm}.
\end{theorem} 
\begin{remark}
Relying on \eqref{Var1} below,
the condition $s_n\geq c_0 n^{\delta_0}$ is satisfied if $\sum_{j=1}^n c_j\geq c_0C_1^{-1}n^{\del_0}$ where
$c_j=\min_{|u|=1}(\text{Cov}(X_j)u\cdot u)=\min_{|u|=1}\text{Var}(X_j\cdot u)$.  
\end{remark}
\begin{proof}[Proof of Theorem \ref{Thm1}]
First,
by \cite[Proposition 13]{Pel}, there are constants $C_1,C_2>0$ so that for all $n,m$ with $n\leq m$ and every unit vector $u$ we have
\begin{equation}\label{Var1}
C_1\sum_{j=n}^{m}\text{Var}(X_j\cdot u)\leq\text{Var}(S_{n,m}\cdot u)\leq C_2\sum_{j=n}^{m}\text{Var}(X_j\cdot u)
\end{equation}
By using \eqref{Var1} and \eqref{Ut} we see that Assumption \ref{AssVV} is valid.
\vskip0.1cm

The proof of Theorem \ref{Thm1} proceeds now similarly to the proof of Theorem \ref{Main Thm}, with the following exception: we cannot use \cite[Theorem 6.17]{PelBook} in order to obtain \eqref{Fin}, since it requires \eqref{phi half}.
In order to overcome this difficulty,
consider first the scalar case $d=1$. Then, along the lines of the proof of \cite[Lemma 2.16]{DS}, it was  shown that for every exponentially fast $\rho$-mixing sequence $\{X_j\}$ which is uniformly bounded by some $K$, for all even $p\geq2$ there exist constants $E_{p,K}>0$ and $V_{p,K}>0$, depending only on $p$ and $K$, so that for all $n$ and $m$ with $n\leq m$ and $\sum_{j=n}^{m}\text{Var}(X_j)\geq V_{p,K}$, we have
\begin{equation}\label{Lp Bounds1}
\|S_{m,n}\|_{L^p}\leq E_{p,K}\Big(\sum_{j=n}^{m}\text{Var}(X_j)\Big)^{1/2}.
\end{equation}
Now, by (\ref{Var1}) we have that 
\[
\sum_{j=n}^{m}\text{Var}(f_j(X_j))\leq C_1^{-1}\text{Var}(S_{n,m})
\]
and so there are constants $R_p,U_p>0$ so that for all $n,m$ with $\|S_{m,n}\|_2\geq U_p$ we have
\begin{equation}\label{Est}
\|S_{n,m}\|_{L^p}\leq R_p\|S_{n,m}\|_{L^2}.
\end{equation}
By replacing $X_j$ with $X_j\cdot u$ for an arbitrary unit vector $u$ and then taking the supremum over $u$, we see that \eqref{Est} holds true also in the vector-valued case (i.e. when $d>1$).

Finally, let us obtain \eqref{Fin}. Set $\mathcal B_n=\sum_{j=1}^{k_n}\Xi_j$. Then by the Markov inequality for every $\ve>0$ and $q>1$ we have
$$
\bbP(|S_n-\mathcal B_n|\geq n^{\ve})=\bbP(|S_n-\mathcal B_n|^q\geq n^{\ve q})\leq n^{-\ve q}\|S_n-\mathcal B_n\|_{L^q}^q\leq R_{q,K}(1+c)n^{-\ve q}
$$
where in the last inequality we have also used \eqref{Est} and that $\|S_n-\cB_n\|_{L^2}\leq c$ is bounded in $n$.
Taking $q>1/\ve$ and applying the Borel-Cantelli lemma we get that 
$$
|S_n-\mathcal B_n|=o(n^\ve)=o(s_n^{\frac{\varepsilon}{
\delta_0}}),\,\text{a.s.}
$$
Since $\varepsilon$ is arbitrary small we get  that  for every $\varepsilon>0$ we have
$$
|S_n-\cB_n|=o(s_n^\varepsilon),\,\,\text{a.s.}
$$
Now the proof of  Theorem \ref{Thm1} is completed similarly to the end of the proof of Theorem \ref{Main Thm}.
\end{proof}

\begin{acknowledgment}
The original rates obtained in previous versions of this paper were  $o(n^\del)+o(V_n^{1/4+\del})$, for any $\del>0$. 
I would like to thank D. Dolgopyat for several discussions which helped improving these rates to the current rates $o(V_n^{1/4+\del})$ in Theorem \ref{Main Thm}.
\end{acknowledgment}


\begin{thebibliography}{Bow75}

\bibliographystyle{alpha}
\itemsep=\smallskipamount


\bibitem{BP}
J. Berkes and W. Philipp, {\em Approximation theorems for independent and weakly
dependent random vectors}, Ann. Probab. 29-54 (1979).

 
\bibitem{BrMix}
R. Bradley, {\em Basic properties of strong mixing conditions. A survey and some open questions}, Probability Surveys, Vol. 2 (2005) 107--144.

\bibitem{Br}
R.C. Bradley, {\em Introduction to Strong Mixing Conditions}, Volume 1, Kendrick Press, Heber City, 2007.


\bibitem{CDF}
C. Cuny, J. Dedecker, F. Merlev\'ede, {Rates of convergence in invariance principles for random
walks on linear groups via martingale methods}, Trans. Amer. Math. Soc. 374 (2021), 137-174.



\bibitem{Douk1}
P. Doukhan, {\em Mixing: Properties and Examples}, 
Lecture Notes in Statistics, Vol. 85, Springer, Berlin (1994).


\bibitem{Dub}
R. Dobrushin, R. {\em Central limit theorems for non-stationary Markov chains I, II}. Theory Probab. Appl.1, 65-80, 329-383 (1956).




\bibitem{DFGTV1} D. Dragi\v cevi\' c, G. Froyland, C. Gonzalez-Tokman and S. Vaienti, \emph{Almost Sure Invariance Principle for random piecewise expanding maps}, Nonlinearity \textbf{31} (2018), 2252-2280.




\bibitem{DS}
D. Dolgopyat, O. Sarig, {\em Local limit theorems for inhomogeneous Markov chains},
https://arxiv.org/abs/2109.05560


\bibitem{DH}
D. Dragi\v cevi\' c, Y. Hafouta, {\em Almost sure invariance principle for random dynamical systems via Gouëzel's approach}, Nonlinearity, 34 6773.






\bibitem{GO} S. Gou\"ezel, \emph{Almost sure invariance principle for dynamical systems by spectral methods},  Annals of Probability  \textbf{38}  (2010), 1639--1671.


\bibitem{Hall}
P.G. Hall and C.C. Hyde, {\em Martingale central limit theory and its application},
Academic Press, New York, 1980.

\bibitem{HR} D. L. Hanson and R. P. Russo, \emph{Some Results on Increments of the Wiener Process with Applications to Lag Sums of I.I.D. Random Variables}, Ann. Probab. \textbf{11} (1983), 609--623.


\bibitem{Hyd}
Nicolai Haydn, Matthew Nicol, Andrew Török and Sandro Vaienti, {\em Almost sure invariance principle for sequential and non-stationary dynamical systems}, Trans. Amer. Math. Soc. 369 (2017), 5293-5316.


\bibitem{IF}
M. Iosifescu and R. Theodorescu, {\em Random processes and learning}, Die Grundlehren der mathematischen Wissenschaften, Band 150. Springer-Verlag, New York (1969).

\bibitem{PelBook}
F. Merlev\'ede,  M. Peligrad, M. and S. Utev, S, {\em Functional Gaussian Approximation for Dependent Structures}, Oxford University Press (2019).


\bibitem{PelASIP}
M. Peligrad and S. Utev, {\em A new maximal inequality and invariance principle for stationary
sequences}. Ann. Probab. 33, 798-815 (2005). 

\bibitem{Pel}
M. Peligrad, {\em Central limit theorem for triangular arrays
of non-homogeneous Markov chains},
Probab. Theory Relat. Fields (2012) 154:409-428.





\bibitem{PS}
W. Philipp and W.F. Stout, {\em Almost sure invariance principles for partial sums of weakly dependent random variables}, Mem. Amer. Math. Sot. 161 (1975).


\bibitem{Rio}
E. Rio, {\em Th\'eorie asymptotique des processus aléatoires faiblement d\'ependants. Math\'ematiques et Applications 31}, Springer-Verlag, Berlin, 2000

\bibitem{Shao}
Q.M. Shao, {\em Almost sure invariance principles for mixing sequences of random variables},
Stochastic Processes and their Applications 48, 319-334 (1993).

\bibitem{VarSeth}
S. Sethuraman and S.R.S Varadhan, {\em A martingale proof of Dobrushin’s theorem for non-homogeneous
Markov chains}, Electron. J. Probab. 10, 1221–1235 (2005).

\bibitem{WZ}
W. Wu and Z. Zou, {\em Gaussian approximations for non-stationary multiple time series
Statistica Sinica}, Vol. 21, No. 3 , pp. 1397-1413 (2011).



\end{thebibliography}
\end{document}